\documentclass[a4paper,12pt]{article}

\usepackage{amsmath,amssymb,amsthm}
\title{The growth rate and dimension theory of beta-expansions}
\author{Simon Baker}

\newtheorem{thm}{Theorem}[section]

\newtheorem{lem}[thm]{Lemma}

\newtheorem{prop}[thm]{Proposition}

\newtheorem{remark}[thm]{Remark}
\newtheorem{property}[thm]{Property}
\begin{document}
\maketitle

\begin{abstract}
In a recent paper of Feng and Sidorov they show that for $\beta\in(1,\frac{1+\sqrt{5}}{2})$ the set of $\beta$-expansions grows exponentially for every $x\in(0,\frac{1}{\beta-1})$. In this paper we study this growth rate further. We also consider the set of $\beta$-expansions from a dimension theory perspective.
\end{abstract}
\section{Introduction}
Let $1<\beta<2$ and $I_{\beta}=[0,\frac{1}{\beta-1}]$. Each $x\in I_{\beta}$ can be expressed as $$x=\sum_{n=1}^{\infty}\frac{\epsilon_{n}}{\beta^{n}},$$ for some $(\epsilon_{n})_{n=1}^{\infty}\in\{0,1\}^{\mathbb{N}}$. We call such a sequence a \textit{$\beta$-expansion} for $x$. We define $$\Sigma_{\beta}(x)=\Big\{(\epsilon_{n})_{n=1}^{\infty}\in\{0,1\}^{\mathbb{N}}:\sum_{n=1}^{\infty} \frac{\epsilon_{n}}{\beta^{n}}=x\Big\}.$$ In $\cite{Erdos}$ it is shown that for $\beta\in(1,\frac{1+\sqrt{5}}{2})$ and $x\in(0,\frac{1}{\beta-1})$ the set $\Sigma_{\beta}(x)$ is uncountable. In \cite{Sidorov,Sidorov2} the author considers the case where $\beta\in[\frac{1+\sqrt{5}}{2},2)$. They show that for Lebesgue almost every $x\in I_{\beta}$ the set $\Sigma_{\beta}(x)$ is uncountable. To describe the growth rate of $\beta$-expansions we consider the following. Let$$\mathcal{E}_{k}(x,\beta)=\left\{(\epsilon_{1},\ldots,\epsilon_{k})\in\{0, 1\}^{k}|\textrm{ } \exists (\epsilon_{k+1}, \epsilon_{k+2}, \ldots)\in \{0,1\}^{\mathbb{N}}: \sum_{n=1}^{\infty}\frac{\epsilon_{n}}{\beta^{n}}=x\right\},$$ we define an element of $\mathcal{E}_{k}(x,\beta)$ to be a \textit{$k$-prefix} for $x$. Moreover, we let $$\mathcal{N}_{k}(x,\beta)=\#\mathcal{E}_{k}(x,\beta)$$and define the \textit{growth rate of $\mathcal{N}_{k}(x,\beta)$} to be $$\lim_{k\to\infty} \frac{\log_{2}\mathcal{N}_{k}(x,\beta)}{k},$$ when this limit exists. When this limit doesn't exist we can consider the \textit{lower and upper growth rates of $\mathcal{N}_{k}(x,\beta)$}, these are defined to be $$\liminf_{k\to\infty} \frac{\log_{2}\mathcal{N}_{k}(x,\beta)}{k}\textrm{ and }\limsup_{k\to\infty} \frac{\log_{2}\mathcal{N}_{k}(x,\beta)}{k}$$ respectively. 

In this paper we also consider $\Sigma_{\beta}(x)$ from a dimension theory perspective. We endow $\{0,1\}^{\mathbb{N}}$ with the metric $d(\cdot,\cdot)$ defined as follows:
\[ d(x,y) = \left\{ \begin{array}{ll}
         2^{-n(x,y)} & \mbox{if $x\neq y,$ where $n(x,y)=\inf \{i:x_{i}\neq y_{i}\} $}\\
        0 & \mbox{if $x=y$.}\end{array} \right. \] We consider the Hausdorff dimension of $\Sigma_{\beta}(x)$ with respect to this metric. It is a simple exercise to show the following inequalities hold:
\begin{equation}
\label{dimension inequality}
\dim_{H}(\Sigma_{\beta}(x))\leq \liminf_{k\to\infty} \frac{\log_{2}\mathcal{N}_{k}(x,\beta)}{k}\leq\limsup_{k\to\infty} \frac{\log_{2}\mathcal{N}_{k}(x,\beta)}{k}.
\end{equation} 
In \cite{FengSid} the following theorem was shown to hold.
\begin{thm}
\label{Nik bounds}
Let $\beta\in (1,\frac{1+\sqrt{5}}{2})$. There exists $\kappa(\beta)>0$ such that $$\liminf_{k\to\infty} \frac{\log_{2}\mathcal{N}_{k}(x,\beta)}{k} \geq \kappa(\beta) \textrm{ for all } x\in\Big(0,\frac{1}{\beta-1}\Big).$$
Here $\kappa(\beta)$ is given explicitly by the formula
\[ \kappa(\beta) = \left\{ \begin{array}{ll}
         \frac{1}{2}\Big(\Big[ \log_{\beta} \Big(\frac{\beta^{2}-1}{1+\beta-\beta^{2}}\Big)\Big] +1\Big)^{-1} & \mbox{if $\beta >\sqrt{2}$};\\
        \frac{1}{2}\Big(\Big[ \log_{\beta} \Big(\frac{1}{\beta-1}\Big)\Big] +1\Big)^{-1} & \mbox{if $\beta\leq \sqrt{2}$}.\end{array} \right. \] 
\end{thm}
The growth rate of $\mathcal{N}_{k}(x,\beta)$ is addressed from the measure theoretic point of view in \cite{Kempton}. The following result is implicit.

\begin{thm}
For almost every $\beta\in (1,2)$ and for almost every $x\in I_{\beta},$ $$\limsup_{k\to\infty} \frac{\log_{2}\mathcal{N}_{k}(x,\beta)}{k}=\log_{2}\Big( \frac{2}{\beta}\Big).$$ Moreover, for almost every $\beta\in(1,\sqrt{2})$ and for almost every $x\in I_{\beta}$ $$\lim_{k\to\infty} \frac{\log_{2}\mathcal{N}_{k}(x,\beta)}{k}=\log_{2} \Big(\frac{2}{\beta}\Big).$$
\end{thm}
We remark that the bounds given in Theorem \ref{Nik bounds} are somewhat weak. We observe that $\kappa(\beta)\to 0$ as $\beta\to 1,$ contrary to what we would expect. As $\beta \to 1$ we would expect the number of $k$-prefixes for $x$ to grow and the growth rate of $\mathcal{N}_{k}(x,\beta)$ to increase. In this paper we show that the following theorem holds. 

\begin{thm}
\label{Hausdorff dimension}
There exists a strictly decreasing sequence $(\omega_{m})_{m=1}^{\infty}$ converging to $1,$ such that, if $\beta\in(1,\omega_{m}]$ then $$\dim_{H}(\Sigma_{\beta}(x))\geq \frac{2m}{2m+1} \textrm{ for all } x\in\Big(0,\frac{1}{\beta -1}\Big).$$
\end{thm}

As an immediate corollary of Theorem \ref{Hausdorff dimension} we have that the infimum of $\dim_{H}(\Sigma_{\beta}(x))\to 1$ as $\beta \to 1,$ i.e.
$$\inf_{x\in (0,\frac{1}{\beta-1})}\left\{\dim_{H}(\Sigma_{\beta}(x))\right\} \to 1$$ as $\beta \to 1$. By (\ref{dimension inequality}) similar statements hold for the lower and upper growth rate of $\mathcal{N}_{k}(x,\beta).$

We also improve on the bounds given in Theorem \ref{Nik bounds}. We show that the following theorem holds.
\begin{thm}
\label{Improved bounds}
There exists a strictly increasing sequence $(\lambda_{m})_{m=1}^{\infty}$ converging to $\frac{1+\sqrt{5}}{2},$ such that, for $\beta\in (1,\lambda_{m}]$ $$\dim_{H}(\Sigma_{\beta}(x)) \geq \frac{1}{m+2} \textrm{ for any } x\in\Big(0,\frac{1}{\beta-1}\Big).$$
\end{thm}

In Section 2 we prove Theorem \ref{Hausdorff dimension} and in Section 3 we prove Theorem \ref{Improved bounds}; Theorem \ref{Improved bounds} will follow by a similar argument to Theorem \ref{Hausdorff dimension}. In Section 4 we give some bounds for the upper growth rate of $\mathcal{N}_{k}(x,\beta)$ and in Section 5 we use our results to obtain bounds for the local dimension of Bernoulli convolutions.

\section{Proof of Theorem \ref{Hausdorff dimension}}
To prove Theorem \ref{Hausdorff dimension} we devise an algorithm for generating $\beta$-expansions. We then show that the Hausdorff dimension of the set of expansions generated by this algorithm is greater than or equal to $\frac{2m}{2m+1}$ for $\beta\in(1,\omega_{m}]$. Before giving details of this algorithm we provide a useful reinterpretation of $\mathcal{N}_{k}(x,\beta)$ and define the sequence $(\omega_{m})_{m=1}^{\infty}$. 

\subsection{$B_{k}(x,\beta)$ and the sequence $(\omega_{m})_{m=1}^{\infty}$}

\subsubsection{Reinterpretation of $\mathcal{N}_{k}(x,\beta)$}
Fix $T_{0,\beta}(x)=\beta x$ and $T_{1,\beta}(x)=\beta x -1.$ We let $$\Omega_{k}=\Big\{a=(a_{n})_{n=1}^{k}\in \{T_{0,\beta},T_{1,\beta}\}^{k}\Big\}.$$ Suppose $x\in  I_{\beta}$ and $a\in\Omega_{k}$, we denote $a_{k}\circ a_{k-1}\circ \ldots \circ a_{1}(x)$ by $a(x)$. We let $$|a|_{0}=\#\Big\{1\leq n\leq k: a_{n}=T_{0,\beta}\Big\}$$ and $$|a|_{1}=\#\Big\{1\leq n\leq k: a_{n}=T_{1,\beta}\Big\}.$$ Finally we define $$T_{k}(x,\beta)=\Big\{a\in\Omega_{k}: a(x)\in I_{\beta}\Big\}$$ and $$B_{k}(x,\beta)=\#T_{k}(x,\beta).$$

\begin{prop}
\label{Change of perspective}
$\mathcal{N}_{k}(x,\beta)=B_{k}(x,\beta)$
\end{prop}
\begin{proof}
Following \cite{FengSid} we observe that
\begin{align*}
\mathcal{E}_{k}(x,\beta)&=\Big\{ (\epsilon_{1},\ldots,\epsilon_{k})\in\{0,1\}^{k}: x-\frac{1}{\beta^{k}(\beta-1)}\leq \sum_{n=1}^{k}\frac{\epsilon_{n}}{\beta^{n}}\leq x\Big\}\\
&=\Big\{ (\epsilon_{1},\ldots,\epsilon_{k})\in\{0,1\}^{k}: 0\leq x-\sum_{n=1}^{k}\frac{\epsilon_{n}}{\beta^{n}}\leq \frac{1}{\beta^{k}(\beta-1)}\Big\}\\
&=\Big\{ (\epsilon_{1},\ldots,\epsilon_{k})\in\{0,1\}^{k}: 0\leq \beta^{k}x-\sum_{n=1}^{k}\epsilon_{n}\beta^{k-n}\leq \frac{1}{\beta-1}\Big\}\\
&=\Big\{ (\epsilon_{1},\ldots,\epsilon_{k})\in\{0,1\}^{k}: 0\leq (T_{\epsilon_{1},\beta},\ldots T_{\epsilon_{k},\beta})(x)\leq \frac{1}{\beta-1}\Big\}.
\end{align*}Our result follows immediately.
\end{proof}
By Proposition \ref{Change of perspective} we can identify elements of $T_{k}(x,\beta)$ with elements of $\mathcal{E}_{k}(x,\beta),$ as such we also define an element of $T_{k}(x,\beta)$ to be a \textit{$k$-prefix} of $x$. To help with our later calculations we include the following technical lemmas.

\begin{lem}
\label{Technical Lemma}
Let $k,n\in\mathbb{N},$ then $$T^{k}_{1,\beta}\Big(\frac{\beta^{n}}{\beta^{2}-1}\Big)=\frac{\beta^{n+k}-\beta^{k+1}-\beta^{k}+\beta+1}{\beta^{2}-1}.$$
\end{lem}
\noindent The proof of this lemma is trival and hence ommited. 
\begin{lem}
\label{Technical Lemma 2}
Assume $a\in\Omega_{2k+1}$ and $|a|_{0}\geq k+1.$ Then for all $x\in \mathbb{R}$ $$a(x)\geq(\overbrace{T_{1,\beta},\ldots,T_{1,\beta},}^{k}\overbrace{T_{0,\beta},\ldots,T_{0,\beta}}^{k+1})(x).$$ Similarly suppose $|a|_{1}\geq k+1$ then $$a(x)\leq(\overbrace{T_{0,\beta},\ldots,T_{0,\beta},}^{k}\overbrace{T_{1,\beta},\ldots,T_{1,\beta}}^{k+1})(x).$$
\end{lem}
\begin{proof}
Suppose $a\in \Omega_{2k+1}$ and $|a|_{0}\geq k+1.$ By a simple calculation we have $$a(x)=\beta^{2k+1}x-\sum_{n=1}^{2k+1} \chi_{n}(a)\beta^{2k+1-n},$$ where $\chi_{n}(a)=1$ if $a_{n}=T_{1,\beta}$ and $0$ otherwise. Since $|a|_{0}\geq k+1$ we have that $\chi_{n}(a)=1$ for at most $k$ different values of $n$. It follows that $$a(x)\geq \beta^{2k+1}x-\beta^{2k}-\ldots - \beta^{k+1}.$$ However,$$\beta^{2k+1}x-\beta^{2k}-\ldots - \beta^{k+1}=(\overbrace{T_{1,\beta},\ldots,T_{1,\beta},}^{k}\overbrace{T_{0,\beta},\ldots,T_{0,\beta}}^{k+1})(x)$$ and our results follows. The second inequality is proved similarly.
\end{proof}
\subsubsection{The sequence $(\omega_{m})_{m=1}^{\infty}$}
We now define our sequence $(\omega_{m})_{m=1}^{\infty}.$ To each $m\in\mathbb{N}$ we associate the polynomials:
\begin{align*}
&P^{1}_{m}(x)=x^{4m+3}-x^{2m+2}-x^{m+2}-x^{m+1}+x +1\\
&P^{2}_{m}(x)=x^{2m+3}-x^{2m+2}-x^{2}+1\\
&P^{3}_{m}(x)=x^{2m+3}-x-1.
\end{align*}
We define $\omega^{(i)}_{m}$ to be the smallest real root of $P^{i}_{m}(x)$ greater than $1$. Clearly $P_{m}^{3}(x)$ has a real root greater than $1$. To see that $P^{1}_{m}(x)$ and $P^{2}_{m}(x)$ both have a real root greater than $1$ we observe that $P^{i}_{m}(1)=0$ and $(P^{i}_{m})'(1)<0$ for $i=1,2$. We define $\omega_{m}=\min\{\omega_{m}^{(1)},\omega_{m}^{(2)},\omega_{m}^{(3)}\}$. We now state some properties of $(\omega_{m})_{m=1}^{\infty}$ that will be useful in our later analysis.

\begin{property}
\label{omega property}
For $\beta\in(1,\omega_{m}]$
$$\beta^{4m+3}-\beta^{2m+2}-\beta^{m+2}-\beta^{m+1}+\beta +1\leq 0.$$ 
\end{property}
\begin{proof}
This follows since $\omega_{m}\leq \omega^{(1)}_{m}$, $P^{1}_{m}(1)=0$ and $(P_{m}^{1})'(1)<0$.
\end{proof}
\begin{property}
\label{omega 2 property}
For $\beta\in(1,\omega_{m}]$ $$T_{0,\beta}^{2m+1}\Big(\frac{\beta}{\beta^{2}-1}\Big)=\frac{\beta^{2m+2}}{\beta^{2}-1}$$ and $\frac{\beta^{2m+2}}{\beta^{2}-1}\in I_{\beta}$. 
\end{property}
\begin{proof}
It suffices to show that $$\frac{\beta^{2m+2}}{\beta^{2}-1}\leq \frac{1}{\beta-1}.$$ This is equivalent to $$\beta^{2m+3}-\beta^{2m+2}-\beta^{2}+1\leq 0.$$ This is true for $\beta\in(1,\omega_{m}]$ since $\omega_{m}\leq \omega^{(2)}_{m},$ $P^{2}_{m}(1)=0$ and $(P_{m}^{2})'(1)<0$.
\end{proof}

\begin{property}
\label{omega 3 property}
For $\beta\in(1,\omega_{m}]$ $$T_{1,\beta}^{2m+1}\Big(\frac{\beta}{\beta^{2}-1}\Big)=\frac{-\beta^{2m+1}+\beta +1}{\beta^{2}-1},$$by Lemma \ref{Technical Lemma}. Moreover $\frac{-\beta^{2m+1}+\beta +1}{\beta^{2}-1}\in I_{\beta}$. 
\end{property}
\begin{proof}It suffices to show that $$\frac{-\beta^{2m+1}+\beta +1}{\beta^{2}-1}\geq 0.$$ This is equivalent to $$\beta^{2m+1}-\beta -1\leq 0.$$ This is true for $\beta\in(1,\omega_{m}]$ since $\omega_{m}\leq \omega^{(3)}_{m}$ and $P_{m}^{3}(1)<0.$
\end{proof}

\begin{property}
The sequence $(\omega_{m})_{m=1}^{\infty}$ is strictly decreasing and converging to $1$. This follows from the fact that $(\omega_{m}^{(i)})_{m=1}^{\infty}$ is strictly decreasing and converging to $1$ for $i=1,2,3$.
\end{property}
\noindent In Section 6 we include a table of values for the sequence $(\omega_{m})_{m=1}^{\infty}.$

\subsection{Algorithm for generating $\beta$-expansions}
We now give details of our algorithm for generating $\beta$-expansions that we mentioned at the start of this section. In what follows we assume $\beta\in (1,\omega_{m}]$ and $x \in (0,\frac{1}{\beta-1}).$ 

We define the interval $\mathcal{I}_{m}$ to be 
\begin{align*}
\mathcal{I}_{m}&=\Big[T^{2m+1}_{1,\beta}\Big(\frac{1}{\beta^{2}-1}\Big),T^{2m+1}_{0,\beta}\Big(\frac{\beta}{\beta^{2}-1}\Big)\Big]\\
&=\Big[\frac{-\beta^{2m+2}+\beta+1}{\beta^{2}-1},\frac{\beta^{2m+2}}{\beta^{2}-1}\Big].
\end{align*}The interval $\mathcal{I}_{m}\subset I_{\beta}$ by Property \ref{omega 2 property} and Property \ref{omega 3 property}. 
\begin{remark}
\label{period 2}
The significance of the points $\frac{1}{\beta^{2}-1}$ and $\frac{\beta}{\beta^{2}-1}$ is that $T_{0,\beta}(\frac{1}{\beta^{2}-1})=\frac{\beta}{\beta^{2}-1}$ and $T_{1,\beta}(\frac{\beta}{\beta^{2}-1})=\frac{1}{\beta^{2}-1}.$ Therefore it is not possible for a point to pass over the interval $[\frac{1}{\beta^{2}-1},\frac{\beta}{\beta^{2}-1}]$ without landing in it.
\end{remark}

\begin{itemize}
\item[Step 1] There exists a minimal number of transformations $j(x)$ that map the point $x$ to the interval $\mathcal{I}_{m}.$ This follows from the fact that $[\frac{1}{\beta^{2}-1},\frac{\beta}{\beta^{2}-1}]\subset \mathcal{I}_{m}$ and Remark \ref{period 2}. We choose a sequence of transformations $a\in\Omega_{j(x)}$ such that $a(x)\in \mathcal{I}_{m}.$  We fix the first $j(x)$ entries in our $\beta$-expansion of $x$ to be those uniquely determined by $a$.
\item[Step 2] If $a(x)\in [\frac{1}{\beta^{2}-1},\frac{\beta^{2m+2}}{\beta^{2}-1}]$ then we can extend the $j(x)$-prefix $a$ to a $(j(x)+2m+1)$-prefix by concatenating $a$ with any element $a^{(1)}\in \Omega_{2m+1}$ such that $|a^{(1)}|_{1}\geq m+1$. To see why $aa^{(1)}$ is a $(j(x)+2m+1)$-prefix for $x$ we observe that 
\begin{align*}
\frac{-\beta^{2m+2}+\beta+1}{\beta^{2}-1}&= T^{2m+1}_{1,\beta}\Big(\frac{1}{\beta^{2}-1}\Big)\\
&\leq aa^{(1)}(x)\\
&\leq aa^{(1)}\Big(\frac{\beta^{2m+2}}{\beta^{2}-1}\Big)\\
&\leq (\overbrace{T_{0,\beta},\ldots,T_{0,\beta},}^{m}\overbrace{T_{1,\beta},\ldots,T_{1,\beta}}^{m+1})\Big(\frac{\beta^{2m+2}}{\beta^{2}-1}\Big)\\
&=T^{m+1}_{1,\beta}\Big(\frac{\beta^{3m+2}}{\beta^{2}-1}\Big)\\
&= \frac{\beta^{4m+3}-\beta^{m+2}-\beta^{m+1}+\beta +1}{\beta^{2}-1},
\end{align*}
by Lemmas \ref{Technical Lemma} and \ref{Technical Lemma 2}. The inequality $$\frac{\beta^{4m+3}-\beta^{m+2}-\beta^{m+1}+\beta +1}{\beta^{2}-1}\leq \frac{\beta^{2m+2}}{\beta^{2}-1},$$ is equivalent to $\beta^{4m+3}-\beta^{2m+2}-\beta^{m+2}-\beta^{m+1}+\beta+1\leq 0,$ which is true for $\beta\in(1,\omega_{m}]$ by Property \ref{omega property}. Therefore $aa^{(1)}(x)\in \mathcal{I}_{m}$ for all $a^{(1)}\in \Omega_{2m+1}$ such that $|a^{(1)}|_{1}\geq m+1$, which by the remarks following Proposition \ref{Change of perspective} implies that $aa^{(1)}$ is a $(j(x)+2m+1)$-prefix for $x$.

If $a(x)\in [\frac{-\beta^{2m+2}+\beta+1}{\beta^{2}-1},\frac{1}{\beta^{2}-1}]$ then we consider elements $a^{(1)}\in \Omega_{2m+1}$ such that $|a^{(1)}|_{0}\geq m+1$. By a similar argument it can be shown that $aa^{(1)}(x)\in \mathcal{I}_{m}$ for all $a^{(1)}\in \Omega_{2m+1}$ such that $|a^{(1)}|_{0}\geq m+1$. As $\#\{a\in \Omega_{2m+1} : |a^{(1)}|_{1}\geq m+1\}=\#\{a\in \Omega_{2m+1} : |a^{(1)}|_{0}\geq m+1\}=2^{2m}$ our algorithm generates $2^{2m}$ $(j(x)+2m+1)$-prefixes for $x$.   
\item[Step 3] We proceed inductively, if $aa^{(1)}(x)\in[\frac{1}{\beta^{2}-1},\frac{\beta^{2m+2}}{\beta^{2}-1}]$ then we extend the $(j(x)+2m+1)$-prefix $aa^{(1)}$ to a $(j(x)+4m+2)$-prefix for $x$ by concatenating $aa^{(1)}$ with any element $a^{(2)}\in \Omega_{2m+1}$ such that $|a^{(2)}|_{1}\geq m+1.$ Similarly, if $aa^{(1)}(x)\in [\frac{-\beta^{2m+2}+\beta+1}{\beta^{2}-1},\frac{1}{\beta^{2}-1}]$ then we consider elements $a^{(2)}\in \Omega_{2m+1}$ such that $|a^{(2)}|_{0}\geq m+1$. We repeat this process indefinitely.
\end{itemize}

\begin{remark}
\label{number of prefix's}
It is clear that by proceeding inductively our algorithm generates $2^{2km}$ $(j(x)+k(2m+1))$-prefixes for $x,$ for $k\in\mathbb{N}$.
\end{remark}

\begin{remark}
The construction of our interval $\mathcal{I}_{m}$ is somewhat arbitrary. We could have begun by choosing any interval of the form $[z,\beta z]$ for some $z\in(0,\frac{1}{\beta-1})$. We then construct the interval $[T^{2m+1}_{1,\beta}(z),T^{2m+1}_{0,\beta}(\beta z)]$ to perform the role of the interval $\mathcal{I}_{m}$. By defining a similar set of polynomials to $P^{1}_{m}(x)$, $ P^{2}_{m}(x)$ and $P^{3}_{m}(x)$ and assuming our $\beta$ satisfies certain restrictions imposed by these polynomials we can ensure that $[T^{2m+1}_{1,\beta}(z),T^{2m+1}_{0,\beta}(\beta z)]\subset I_{\beta}$ and if $x\in[T^{2m+1}_{1,\beta}(z),T^{2m+1}_{0,\beta}(\beta z)]$ then $a(x)\in [T^{2m+1}_{1,\beta}(z),T^{2m+1}_{0,\beta}(\beta z)]$ for $2^{2m}$ elements of $\Omega_{2m+1}$. It would be interesting to know whether $\mathcal{I}_{m}$ is the most efficient choice of interval for this method.
\end{remark}

We denote the set of $\beta$-expansions generated by this algorithm by $\Sigma_{\beta}(x,m)$ and the set of $k$-prefixes for $x$ generated by this algorithm by $\Sigma_{\beta}(x,m,k).$ 
\begin{property}
\label{Increasing prefix's}
The cardinality of $\Sigma_{\beta}(x,m,k)$ is increasing with respect to $k$.
\end{property}
\begin{property}
\label{bridge property}
Let $s,s'\in\mathbb{N}$ and $s>s'$. If $b\in \Sigma_{\beta}(x,m,j(x)+s'(2m+1))$ then
\begin{align*}
\#\Big\{a\in \Sigma_{\beta}(x,m,j(x)+s(2m+1))&: a_{n}=b_{n} \textrm{ for } 1\leq n\leq j(x)+s'(2m+1)\Big\}\\
&=2^{2m(s-s')}.
\end{align*}
\end{property}We now prove several technical lemmas.

\begin{lem}
\label{Minimal prefix's}
Let $\beta\in (1,\omega_{m}]$ and $x \in (0,\frac{1}{\beta-1}).$ Assume $k\geq j(x)$ then $$\#\Sigma_{\beta}(x,m,k)\geq 2^{2m(\frac{k-j(x)}{2m+1}-1)}.$$
\end{lem}
\begin{proof}
We have that 
\begin{align*}
\#\Sigma_{\beta}(x,m,k)\geq \#\Sigma_{\beta}\Big(x,m,\Big((2m+1)\Big[\frac{k-j(x)}{2m+1}\Big] +j(x)\Big)\Big)&\geq 2^{2m\Big[\frac{k-j(x)}{2m+1}\Big]}\\
&\geq 2^{2m(\frac{k-j(x)}{2m+1}-1)}.
\end{align*}
\end{proof}
\begin{lem}
\label{Maximal prefix's}
Let $l\geq j(x)$ and $b\in \Sigma_{\beta}(x,m,l),$ then for $k\geq l$ $$\#\{a=(a_{n})_{n=1}^{k}\in \Sigma_{\beta}(x,m,k): a_{n}=b_{n}\textrm{ for } 1\leq n \leq l\}\leq 2^{2m(\frac{k-l}{2m+1}+2)}.$$
\end{lem}
\begin{proof}
Consider the integers $j(x)+(2m+1)[\frac{l-j(x)}{2m+1}]$ and $j(x)+ (2m+1)([\frac{k-j(x)}{2m+1}]+1)$. We remark that $j(x)+(2m+1)[\frac{l-j(x)}{2m+1}]$ is an integer of the form $j(x)+s(2m+1)$ and is less than or equal to $l$ and $j(x)+ (2m+1)([\frac{k-j(x)}{2m+1}]+1)$ is an integer of the form $j(x)+s(2m+1)$ greater that or equal to $k$. It follows immediately that 
\begin{align*}
&\phantom{hs.}\#\Big\{a=(a_{n})_{n=1}^{k}\in \Sigma_{\beta}(x,m,k): a_{n}=b_{n} \textrm{ for } 1\leq n \leq l\Big\}\\
&\leq \#\Big\{a=(a_{n})_{n=1}^{j(x)+ (2m+1)([\frac{k-j(x)}{2m+1}]+1)}\ldots\\
&\phantom{jdjda} \in \Sigma_{\beta}(x,m,j(x)+ (2m+1)(\Big[\frac{k-j(x)}{2m+1}\Big]+1)): a_{n}=b_{n}\ldots \\
&\phantom{dfsd0}\textrm{ for } 1\leq n \leq j(x)+(2m+1)\Big[\frac{l-j(x)}{2m+1}\Big]\Big\}\\
&\leq 2^{2m([\frac{k-j(x)}{2m+1}]+1-[\frac{l-j(x)}{2m+1}])}\\
&\leq 2^{2m(\frac{k-j(x)}{2m+1} +1-\frac{l-j(x)}{2m+1}+1)}\\
&=2^{2m(\frac{k-l}{2m+1}+2)},
\end{align*}by Properties \ref{Increasing prefix's} and \ref{bridge property}.
\end{proof}

\subsubsection{Proof of Theorem \ref{Hausdorff dimension}}
We are now in a position to prove Theorem \ref{Hausdorff dimension}. The following proof is based upon the argument given in Example $2.7$ of \cite{Falconer}. 

\begin{proof}[Proof of Theorem \ref{Hausdorff dimension}]
By the monotonicity of Hausdorff dimension with respect to inclusion it suffices to show that $\dim_{H}(\Sigma_{\beta}(x,m))\geq\frac{2m}{2m+1}.$ It suffices to show that for any sufficiently small cover $\{U_{i}\}_{i=1}^{\infty}$ of $\Sigma_{\beta}(x,m)$ we can bound  $\sum_{i=1}^{\infty}\textrm{Diam}(U_{i})^{\frac{2m}{2m+1}}$ below by a positive constant independent of our choice of cover. It is a simple exercise to show that $\Sigma_{\beta}(x,m)$ is a compact set; by this result we may restrict to finite covers of $\Sigma_{\beta}(x,m).$ Let $\{U_{i}\}_{i=1}^{N}$ be a finite cover of $\Sigma_{\beta}(x,m).$ Without loss of generality we may assume that all elements of our cover satisfy $\textrm{Diam}(U_{i})< 2^{-j(x)}$. For each $U_{i}$ there exists $l(i)$ such that $$2^{-(l(i)+1)}\leq \textrm{Diam}(U_{i})< 2^{-l(i)}.$$ It follows that there exists $z^{(i)}\in \{0,1\}^{l(i)}$ such that $y_{n}=z^{(i)}_{n}$ for all $y\in U_{i},$ for $1\leq n \leq l(i)$. We may assume that $z^{(i)}\in \Sigma_{\beta}(x,m,l(i)),$ if we supposed otherwise then $\Sigma_{\beta}(x,m)\cap U_{i}=\emptyset$ and we can remove $U_{i}$ from our cover. We denote by $C_{i}$ the set of sequences in $\{0,1\}^{\mathbb{N}}$ whose first $l(i)$ entries agree with $z^{(i)},$ i.e. $$C_{i}=\Big\{(\epsilon_{n})_{n=1}^{\infty}\in \{0,1\}^{\mathbb{N}}: \epsilon_{n}= z^{(i)}_{n}\textrm{ for } 1\leq n\leq l(i)\Big\}.$$ Clearly $U_{i}\subset C_{i}$ and therefore the set $\{C_{i}\}_{i=1}^{N}$ is a cover of $\Sigma_{\beta}(x,m).$

Since there are only finitely many elements in our cover there exists $J$ such that $2^{-J}\leq \textrm{Diam}(U_{i})$ for all $i$. We now consider the set $\Sigma_{\beta}(x,m,J).$ Since $\{C_{i}\}_{i=1}^{N}$ is a cover of $\Sigma_{\beta}(x,m)$ each $a\in\Sigma_{\beta}(x,m,J)$ satisfies $a_{n}=z^{(i)}_{n}$ for $1\leq n \leq l(i),$ for some $i$. Therefore $$\# \Sigma_{\beta}(x,m,J)\leq \sum_{i=1}^{N} \#\{a\in \Sigma_{\beta}(x,m,J): a_{n}=z^{(i)}_{n} \textrm{ for } 1\leq n\leq l(i)\}.$$

By counting elements of $\Sigma_{\beta}(x,m,J)$ and Lemmas \ref{Minimal prefix's} and \ref{Maximal prefix's} the following holds;
\begin{align*}
2^{2m(\frac{J-j(x)}{2m+1}-1)}&\leq \# \Sigma_{\beta}(x,m,J)\\
&\leq \sum_{i=1}^{N} \#\{a\in \Sigma_{\beta}(x,m,J): a_{n}=z^{(i)}_{n} \textrm{ for } 1\leq n\leq l(i)\}
\end{align*}
\begin{align*}
&\leq \sum_{i=1}^{N} 2^{2m(\frac{J-l(i)}{2m+1}+2)}\\
&= 2^{\frac{2mJ+1}{2m+1}+4m}\sum_{i=1}^{N} 2^{\frac{-2m(l(i)+1)}{2m+1}}\\
&\leq 2^{\frac{2mJ+1}{2m+1}+4m}\sum_{i=1}^{N} \textrm{Diam}(U_{i})^{\frac{2m}{2m+1}}.
\end{align*} Dividing through by $2^{\frac{2mJ+1}{2m+1}+4m}$ yields $$\sum_{i=1}^{N} \textrm{Diam}(U_{i})^{\frac{2m}{2m+1}} \geq 2^{\frac{-12m^{2}-(6+2j(x))m-1}{2m+1}},$$ which is a positive constant greater than zero that does not depend on our choice of cover. Our result follows.
\end{proof}

\section{Proof of Theorem \ref{Improved bounds}}
Our proof of Theorem \ref{Improved bounds} is analogous to our proof of Theorem \ref{Hausdorff dimension}, as such we only give details where appropriate. Similarly we devise an algorithm for generating $\beta$-expansions, the Hausdorff dimension of the set of expansions generated by this algorithm will be greater than $\frac{1}{m+2}$ for $\beta\in(1,\lambda_{m}]$.

\subsection{The sequence $(\lambda_{m})_{m=1}^{\infty}$}
We now give details of the sequence $(\lambda_{m})_{m=1}^{\infty}$ . Let $\lambda_{m}$ be the smallest real root of the equation  $$x^{m+3}-x^{m+2}-x^{m+1}+1=0$$ greater than $1$. The sequence $(\lambda_{m})_{m=1}^{\infty}$ is well known, see \cite{Bertin} for details. We now make several remarks.

\begin{remark}
To see that $x^{m+3}-x^{m+2}-x^{m+1}+1$ has a real root greater than $1,$ let $P_{m}(x)=x^{m+3}-x^{m+2}-x^{m+1}+1.$ We observe that $P_{m}(1)=0$ and $P_{m}'(1)<0,$ this implies that $P_{m}(x)$ has a real root greater than $1$.
\end{remark}
\begin{property}
\label{lambda property}
For $\beta\in (1,\lambda_{m}]$ we have 
$$\beta^{m+3}-\beta^{m+2}-\beta^{m+1}+1\leq 0.$$ 
\end{property}
\begin{remark}
Each term in $(\lambda_{m})_{m=1}^{\infty}$ is a Pisot number, i.e. a real algebraic integer greater than $1$ whose Galois conjugates are of modulus strictly less than $1$.  Moreover, $\lambda_{1}$ is the greatest real root of $x^{3}-x-1=0,$ the first Pisot number.
\end{remark}
\begin{property}
The sequence $(\lambda_{m})_{m=1}^{\infty}$ is strictly increasing and it is a simple exercise to show that $\lambda_{m}\to \frac{1+\sqrt{5}}{2}$ as $m\to\infty$.
\end{property}
\noindent In Section 6 we include a table of values for the sequence $(\lambda_{m})_{m=1}^{\infty}$.

\subsection{Algorithm for generating $\beta$-expansions}
We define $$\mathcal{I}=\Big[T_{1,\beta}\Big(\frac{1}{\beta^{2}-1}\Big),T_{0,\beta}\Big(\frac{\beta}{\beta^{2}-1}\Big)\Big]=\Big[\frac{1+\beta-\beta^{2}}{\beta^{2}-1},\frac{\beta^{2}}{\beta^{2}-1}\Big],$$ for $1<\beta<\frac{1+\sqrt{5}}{2}$ the interval $\mathcal{I}$ is contained within $I_{\beta}$. This interval will play a similar role to the interval $\mathcal{I}_{m}$. Before giving details of our algorithm we require the following technical Lemma.

\begin{lem}
\label{Tech lemma}
For $\beta\in(1,\lambda_{m}]$ $$T^{m+1}_{0,\beta}\Big(\frac{1+\beta-\beta^{2}}{\beta^{2}-1}\Big) \geq \frac{1}{\beta^{2}-1} \textrm{ and } T^{m+1}_{1,\beta}\Big(\frac{\beta^{2}}{\beta^{2}-1}\Big) \leq \frac{\beta}{\beta^{2}-1}.$$
\end{lem}
\begin{proof}
It is a simple exercise to show that $$T^{m+1}_{0,\beta}\Big(\frac{1+\beta-\beta^{2}}{\beta^{2}-1}\Big)=\frac{\beta^{m+1}+\beta^{m+2}-\beta^{m+3}}{\beta^{2}-1}.$$ This is greater than or equal to $\frac{1}{\beta^{2}-1}$ precisely when $\beta^{m+3}-\beta^{m+2}-\beta^{m+1}+1\leq 0,$ this is true by property \ref{lambda property}. Our second inequality is proved similarly.

\end{proof}

We now formalise our algorithm for generating expansions.

\begin{itemize}
\item[Step 1]
Let $x\in(0,\frac{1}{\beta-1}),$ by Remark \ref{period 2} there exists a minimal number of transformations $g(x)$ that map $x$ into the interval $\mathcal{I}$. We choose a sequence of transformations $a\in \Omega_{g(x)}$ such that $a(x)\in \mathcal{I}$.  We fix the first $g(x)$ entries in our $\beta$-expansion to be those uniquely determined by $a$.

\item[Step 2]If $a(x)\in [\frac{1}{\beta^{2}-1},\frac{\beta}{\beta^{2}-1}]$ then we can extend $a(x)$ to a $(g(x)+1)$-prefix by either $T_{0,\beta}$ or $T_{1,\beta}$, we then choose an element $a^{(i)}(x)\in\Omega_{m+1}$ such that $a^{(i)}\circ T_{i,\beta} \circ a(x)\in \mathcal{I},$ for $i=0,1.$ This defines $2$ prefixes of length $(g(x)+m+2)$ for $x$. If $a(x)\in  [\frac{1+\beta-\beta^{2}}{\beta^{2}-1},\frac{1}{\beta^{2}-1}]$ we iterate $T_{0,\beta}$ until $T^{k}_{0,\beta}\circ a(x)\in [\frac{1}{\beta^{2}-1},\frac{\beta}{\beta^{2}-1}].$ By Lemma \ref{Tech lemma} and the monotonicity of the transformation $T_{0,\beta}$ we have that $k\leq m+1$. The transformation $T^{k}_{0,\beta}\circ a(x)$ defines a $(g(x)+k)$-prefix. We can extend this prefix to a $(g(x)+k+1)$-prefix by either $T_{0,\beta}$ or $T_{1,\beta},$ we then make a choice of element $a^{(i)}\in\Omega_{m+1-k}$ such that $a^{(i)}\circ T_{i,\beta}\circ T^{k}_{0,\beta}\circ a(x)\in \mathcal{I}$. This defines $2$ prefixes of length $(g(x)+m+2)$ for $x$. If $a(x)\in [\frac{\beta}{\beta^{2}-1},\frac{\beta^{2}}{\beta^{2}-1}]$ then by a similar argument to the case where $a(x)\in [\frac{1+\beta-\beta^{2}}{\beta^{2}-1},\frac{1}{\beta^{2}-1}]$ we can formalise a method for choosing two elements of $a^{(0)},a^{(1)}\in \Omega_{m+2}$ such that $a^{(i)}\circ a(x) \in \mathcal{I},$ for $i=0,1.$ At this stage our algorithm has generated $2$ prefixes of length $(g(x)+m+2)$ for $x$.

\item[Step 3]
The $2$ prefixes of length $(g(x)+m+2)$ defined by Step 2 map $x$ into $\mathcal{I},$ as such we can apply Step 2 to the image of $x$ under the transformations corresponding to our $2$ prefixes of length $(g(x)+m+2)$. This defines $4$ prefixes of length $(g(x)+2(m+2))$. We repeat this process indefinitely.
\end{itemize}

\begin{remark}
Proceeding inductively our algorithm generates $2^{k}$ prefixes of length $(g(x)+ k(m+2))$ for $x,$ for all $k\in\mathbb{N}$
\end{remark} Repeating the arguments given in Section 2 we can show that analogues of Property \ref{bridge property} , Lemma \ref{Minimal prefix's} and Lemma \ref{Maximal prefix's} all hold. Theorem \ref{Improved bounds} then follows by an analogous argument to the one given in the proof of Theorem \ref{Hausdorff dimension}.

\begin{remark}
As in the proof of Theorem \ref{Hausdorff dimension} our choice of interval $\mathcal{I}$ is somewhat arbitrary. It would be interesting to know whether $\mathcal{I}$ is the most efficient choice of interval for this method.
\end{remark}
\section{Upper bounds for the upper growth rate of $\mathcal{N}_{k}(x,\beta)$}
In this section we give bounds for the upper growth rate of $\mathcal{N}_{k}(x,\beta).$ Our main result in this section is the following theorem.
\begin{thm}
\label{Upper bounds for beta close to 2}
The supremum of the upper growth rates converges to $0$ as $\beta \to 2,$ i.e., $$\sup_{x\in (0,\frac{1}{\beta-1})}\left\{\limsup_{k\to\infty} \frac{\log_{2}\mathcal{N}_{k}(x,\beta)}{k}\right\} \to 0$$ as $\beta \to 2.$
\end{thm}By $(\ref{dimension inequality})$ similar statements hold for $\dim_{H}(\Sigma_{\beta}(x))$ and the lower growth rate of $\mathcal{N}_{k}(x,\beta)$. Theorem \ref{Upper bounds for beta close to 2} can be interpreted as an analogue of Theorem \ref{Hausdorff dimension} in the case where $\beta$ is close to $2.$ 
\begin{proof}[Proof of Theorem \ref{Upper bounds for beta close to 2}]
Fix $m\in\mathbb{N}$. Recall that 
\begin{equation}
\label{prefix reformulation}
\mathcal{N}_{m}(x,\beta)=\#\Big\{(\epsilon_{1},\ldots,\epsilon_{m})\in\{0,1\}^{m}: x-\frac{1}{\beta^{m}(\beta-1)}\leq \sum_{n=1}^{m}\frac{\epsilon_{n}}{\beta^{n}}\leq x\Big\}.
\end{equation} Let $$L(m,\beta)=\Big\{\sum_{n=1}^{m}\frac{\epsilon_{n}}{\beta^{n}}: (\epsilon_{1},\ldots,\epsilon_{m})\in\{0,1\}^{m}\Big\},$$ $L(m,2)$ is the set of dyadic rationals of degree $m,$ therefore $|x-y|\geq 2^{-m}$ for all $x,y\in L(m,2)$ such that $x\neq y$. By continuity, for each $m\in\mathbb{N}$ there exists $\delta(m)>0,$ such that, for all $\beta \in (2-\delta(m),2)$ $$|x-y|>\frac{1}{2\beta^{m}(\beta-1)},$$ for all $x,y\in L(m,\beta)$ such that $x\neq y$.

It follows that any interval of length $\frac{1}{\beta^{m}(\beta-1)}$ contains at most two elements of $L(m,\beta)$. By the reformulation of $\mathcal{N}_{m}(x,\beta)$ given by (\ref{prefix reformulation}) it follows that for $\beta \in (2-\delta(m),2)$ any $x\in(0,\frac{1}{\beta-1})$ has at most $2$ prefixes of length $m$. Proceeding inductively we can deduce that for $\beta \in (2-\delta(m),2)$ and $x\in(0,\frac{1}{\beta-1})$ $$\mathcal{N}_{km}(x,\beta)\leq 2^k,$$ for all $k\in\mathbb{N}.$ By a simple argument it follows that $$\limsup_{k\to\infty} \frac{\log_{2}\mathcal{N}_{k}(x,\beta)}{k}\leq \frac{1}{m}.$$ As $m$ was arbitrary we can conclude our result.

\end{proof}
The following result gives an upper bound for the upper growth rate of $\mathcal{N}_{k}(x,\beta)$ for $\beta$ close to $1$.

\begin{thm}
\label{Upper bounds for beta close to 1}
Let $m\in \mathbb{N}$ and $m\geq 2$. For $\beta\in(2^\frac{1}{m},2),$ $$\limsup_{k\to\infty} \frac{\log_{2}\mathcal{N}_{k}(x,\beta)}{k}\leq \frac{\log_{2}(2^{m}-1)}{m}\textrm{ for all } x\in\Big(0,\frac{1}{\beta-1}\Big)$$
\end{thm}

\begin{proof}
It is a simple exercise to show that $$T^{-m}_{1,\beta}(0)= \frac{\beta^{m}-1}{\beta^{m}(\beta-1)} \textrm{ and } T^{-m}_{0,\beta}\Big(\frac{1}{\beta-1}\Big)= \frac{1}{\beta^{m }(\beta-1)}.$$ By a simple manipulation $T^{-m}_{1,\beta}(0)>T^{-m}_{0,\beta}(\frac{1}{\beta-1})$ is equivalent to $\beta^{m}>2.$
Let $\beta\in(2^{\frac{1}{m}},2)$ and $x\in (0,\frac{1}{\beta-1}),$ then by the above and the monotonicity of the maps $T_{0,\beta}$ and $T_{1,\beta}$ either $T^{m}_{0,\beta}(x)$ or $T^{m}_{1,\beta}(x)$ will lie outside the interval $I_{\beta}$. It follows that any $x\in (0,\frac{1}{\beta-1})$ can have at at most $(2^{m}-1)$ $m$-prefixes. By an inductive argument it follows that $$N_{km}(x,\beta)\leq (2^{m}-1)^{k},$$ for all $k\in \mathbb{N}$. Our result follows immediately.
\end{proof}

\section{Application to Bernoulli convolutions}
Given $1<\beta<2$ we define the \textit{Bernoulli convolution} $\mu_{\beta}$ as follows. Let $E\subset \mathbb{R}$ be a Borel set, $$\mu_{\beta}(E)=\mathbb{P}\Big(\Big\{ (a_{1},a_{2},\ldots)\in\{0,1\}^{\mathbb{N}}: \sum_{n=1}^{\infty}\frac{a_{n}}{\beta^{n}}\in E\Big\}\Big),$$ where $\mathbb{P}$ is the $(1/2,1/2)$ Bernoulli measure. For $x\in I_{\beta}$ we define the \textit{local dimension} of $\mu_{\beta}$ at $x$ by $$d(\mu_{\beta},x)=\lim_{r\to 0} \frac{\log \mu_{\beta}([x-r,x+r])}{\log r},$$ when this limit exists. When the limit doesn't exists we can consider the \textit{lower and upper local dimension} of $\mu_{\beta}$ at $x$, these are defined as $$\underline{d}(\mu_{\beta},x)=\liminf_{r\to 0} \frac{\log \mu_{\beta}([x-r,x+r])}{\log r},$$ and $$\overline{d}(\mu_{\beta},x)=\limsup_{r\to 0} \frac{\log \mu_{\beta}([x-r,x+r])}{\log r}$$ respectively. In \cite{FengSid} the authors show that the following result holds. 
\begin{thm}
\label{Nik local dimension}
For any $\beta\in(1,\frac{1+\sqrt{5}}{2})$ we have $$\overline{d}(\mu_{\beta},x)\leq (1-\kappa(\beta))\log_{\beta}2,$$ for all $x\in (0,\frac{1}{\beta-1}),$ where $\kappa(\beta)$ is as in Theorem \ref{Nik bounds}.
\end{thm} Replicating the arguments given in \cite{FengSid} and using the improved bounds given by Theorems \ref{Hausdorff dimension} and \ref{Improved bounds} we can show that the following result holds.

\begin{thm}
\label{Improved local dimension}
Let $\beta\in (1,\omega_{m}],$ then $$\overline{d}(\mu_{\beta},x)\leq \frac{1}{2m+1}\log_{\beta}2,$$ for every $x\in (0,\frac{1}{\beta-1}).$ Similarly for $\beta\in(1,\lambda_{m}],$ then $$\overline{d}(\mu_{\beta},x)\leq \frac{m+1}{m+2}\log_{\beta}2,$$ for every $x\in (0,\frac{1}{\beta-1}).$
\end{thm}

\section{Open questions and tables for $(\omega_{m})_{m=1}^{\infty}$ and $(\lambda_{m})_{m=1}^{\infty}$}
Here are a few open questions:
\begin{itemize}
\item Does the positivity of the lower growth rate of $\mathcal{N}_{k}(x,\beta)$ imply the Hausdorff dimension of $\Sigma_{\beta}(x)$ is positive?
\item Do we have equality in (\ref{dimension inequality})?
\item Under what conditions do we have equality in (\ref{dimension inequality})?
\item Is our choice of interval $\mathcal{I}_{m}$ in the proof of Theorem \ref{Hausdorff dimension} the most efficient?
\item Is our choice of interval $\mathcal{I}$ in the proof of Theorem \ref{Improved bounds} the most efficient?
\end{itemize}
The following tables list certain values of $\omega_{m}$ and $\lambda_{m}$ and their associated polynomials.
\begin{table}[ht]
\caption{Table of values for the sequence $(\omega_{m})_{m=1}^{\infty}$} % title of Table
\centering  % used for centering table
\begin{tabular}{c c c  } % centered columns (4 columns)
\hline\hline                        %inserts double horizontal lines
$m$ & $\omega_{m}$(To 5DP) & Associated polynomials$$\\ [0.5ex] % inserts table 
%heading
\hline                  % inserts single horizontal line
1 & 1.07445 & $P^{1}_{1}(x)=x^{7}-x^{4}-x^{3}-x^{2}+x+1 $\\ & & $P^{2}_{1}(x)=x^{5}-x^{4}-x^{2}+1$\\ & & $P^{3}_{1}(x)=x^{5}-x-1$\\ % inserting body of the table
2 & 1.02838 & $P^{1}_{2}(x)=x^{11}-x^{6}-x^{4}-x^{3}+x+1$    \\ & & $P^{2}_{2}(x)=x^{7}-x^{6}-x^{2}+1$\\ & & $P^{2}_{k}(x)=x^{7}-x-1$\\
3 & 1.01492 & $P^{1}_{3}(x)=x^{15}-x^{8}-x^{5}-x^{4}+x+1$  \\ & & $P^{2}_{3}(x)=x^{9}-x^{8}-x^{2}+1$\\ & & $P^{3}_{3}(x)=x^{9}-x-1$\\
10 & 1.00172 & $P^{1}_{10}(x)=x^{43}-x^{22}-x^{12}-x^{11}+x+1$   \\ & & $P^{2}_{10}=x^{23}-x^{22}-x^{2}+1$\\ & & $P^{3}_{10}(x)=x^{23}-x-1$\\
100 & 1.00003 & $P^{1}_{100}(x)=x^{403}-x^{202}-x^{102}-x^{101}+x+1$   \\ & & $P^{2}_{100}(x)=x^{203}-x^{202}-x^{2}+1$\\ & & $P^{3}_{100}(x)=x^{203}-x-1$\\ [1ex]      % [1ex] adds vertical space
\hline %inserts single line
\end{tabular}
\label{table 1} % is used to refer this table in the text
\end{table}
\begin{table}[ht]
\caption{Table of values for the sequence $(\lambda_{m})_{m=1}^{\infty}$} % title of Table
\centering  % used for centering table
\begin{tabular}{c c c} % centered columns (4 columns)
\hline\hline                        %inserts double horizontal lines
$m$ & $\lambda_{m}$(To 5DP) & Associated polynomial \\ [0.5ex] % inserts table 
%heading
\hline                  % inserts single horizontal line
1 & 1.32472(First Pisot number) & $x^{4}-x^{3}-x^{2}+1$  \\ % inserting body of the table
2 & 1.46557 & $x^{5}-x^{4}-x^{3}+1$    \\
3 & 1.53416 & $x^{6}-x^{5}-x^{4}+1$    \\
10 & 1.61575 & $x^{13}-x^{12}-x^{11}+1$   \\
100 & 1.61804 & $x^{103}-x^{102}-x^{101}+1$  \\ [1ex]      % [1ex] adds vertical space
\hline %inserts single line
\end{tabular}
\label{table 2} % is used to refer this table in the text
\end{table}
\newpage

\noindent \textbf{Acknowledgements} The author is indebted to Nikita Sidorov for much encouragement and guidance.

\noindent \textbf{Address} The University of Manchester, School of Mathematics, Oxford Road, Manchester, M13 9PL.


\begin{thebibliography}{100}
\bibitem{Bertin} J. Bertin, A. Decomps-Guilloux, M. Grandet-Hugot, M. Pathiaux-Delefosse and J. P. Schreiber. Pisot
and Salem Numbers. Birkhäuser, Basel, 1992.
\bibitem{Erdos} P. Erd\H{o}s, I. Jo\'{o} and V. Komornik, Characterization of the unique expansions $1 =\sum_{i=1}^{\infty}q^{-n_{i}}$ and
related problems, Bull. Soc. Math. Fr. 118 (1990), 377--390.
\bibitem{Falconer} K. Falconer. Fractal Geometry: Mathematical Foundation and Applications. John Wiley, Chichester, 1990.
\bibitem{FengSid} D. J. Feng, N. Sidorov, Growth rate for beta-expansions, Monatsh. Math. 162 (2011), no. 1, 41--60. 
\bibitem{Kempton} T. Kempton, Counting $\beta$-expansions and the absolute continuity of Bernoulli convolutions, Preprint.
\bibitem{Sidorov} N. Sidorov, Almost every number has a continuum of $\beta$-expansions, Amer. Math. Monthly 110 (2003), 838--842.
\bibitem{Sidorov2} N. Sidorov, Combinatorics of linear iterated function systems with overlaps, Nonlinearity 20 (2007), 1299--1312.
\end{thebibliography}
\end{document}